\documentclass[12]{amsart}

\usepackage{geometry}
\usepackage{amsmath,amssymb,}
\newtheorem{thm}{Theorem}[section]

\newtheorem{cor}[thm]{Corollary}
\newtheorem{lem}[thm]{Lemma}

\newtheorem{prop}[thm]{Proposition}
\newtheorem{exa}[thm]{Example}

\theoremstyle{definition}

\newtheorem{rem}[thm]{Remark}

\numberwithin{equation}{section}

\begin{document}
\title[matrices over a commutative ring]{Matrices over a commutative ring as sums of three idempotents or three involutions}

\author[Tang, Zhou and Su]{Gaohua Tang, Yiqiang Zhou and Huadong Su}

\address{College of Sciences, Qinzhou University,  Qinzhou, Guangxi, P.R. China;
School of Mathematics and Statistics, Guangxi Teachers Education University, Nanning, Guangxi, P.R. China}
\email{tanggaohua@163.com}

\address{Department of Mathematics and Statistics, Memorial University of Newfoundland, St.John's, NL A1C 5S7, Canada}
\email{zhou@mun.ca}

\address{School of Mathematics and Statistics, Guangxi Teachers Education University, Nanning, Guangxi, P.R. China}
\email{huadong.su@163.com}


\baselineskip=20pt
\begin{abstract} 
Motivated by Hirano-Tominaga's work \cite{HT} on rings for which every element is a sum of two idempotents and by 
de Seguins Pazzis's results \cite{de} on decomposing every matrix over a field of positive characteristic as a sum of idempotent matrices, we address decomposing every matrix over a commutative ring as a sum of three idempotent matrices and, respectively, as a sum of  three involutive matrices. 
\end{abstract}
%\date{September 9, 2015}
%\date{\today}
%\today
\maketitle

\bigskip

\section{Introduction}
In this paper, we discuss when every matrix in the matrix ring $\mathbb M_n(R)$ over a commutative ring $R$ is a sum of three idempotent matrices or of three involutive matrices. The motivation comes from two sources. On the one hand, Hirano-Tominaga \cite{HT} showed that, 
in any proper matrix ring, not every matrix is a sum of two idempotents. On the other hand, de Seguins Pazzis's \cite{de} proved  
that, for a field $F$ and an integer $n\ge 1$, if $|F|\le 3$, then every matrix in $\mathbb M_n(F)$ is a sum of three idempotents.
Therefore, $k=3$ is the smallest positive integer such that, for some ring $R$, every matrix over $R$ is a sum of $k$ idempotents, and one may ask: for which rings $R$, is every matrix over $R$ a sum of three idempotents? Here we address this question for matrices over a commutative ring. The following are the main results. The first one shows that the converse of  de Seguins Pazzis's result holds.

\begin{thm}
Let $F$ be a field and $n\ge 1$. The following are equivalent:
\begin{enumerate}
\item Every matrix in ${\mathbb M}_n(F)$ is a sum of three idempotents.
\item Every invertible matrix in ${\mathbb M}_n(F)$ is a sum of three idempotents.
\item $F\cong \mathbb Z_2$ or $F\cong \mathbb Z_3$.
\end{enumerate}
\end{thm} 

If the underling ring is a commutative ring, we obtain:

\begin{thm} Suppose that every matrix in $\mathbb M_n(R)$ is a sum of three idempotents where 
$R$ is a commutative ring and $n\ge 1$.  Then $J(R)$ is nil and $R/J(R)$ has identity $x^3=x$.
If in addition $R$ is an indecomposable ring, then  $R\cong \mathbb Z_n$ where $n=2,3,$ or $4$.
\end{thm}

\begin{thm}
Let $R$ be a commutative ring with ${\rm Nil}(R)=0$ {\rm (}e.g., $J(R)=0${\rm )} and $n\ge 1$.  The following are equivalent:
\begin{enumerate}
\item Every matrix in $\mathbb M_n(R)$ is a sum of three idempotents.

%\item Every invertible matrix in $\mathbb M_n(R)$ is a sum of three idempotents.
\item $R\cong A\times B$, where $A$ is a Boolean ring and $B$ is zero or a subdirect product of $\mathbb Z_3$'s.
\item $R$ has the identity $x^3=x$.
\end{enumerate}
 
\end{thm}

As an application of Theorems 1.1-1.3, one can determine commutative rings over which every matrix is a sum of three involutions.  
\begin{thm}
Let $R$ be a commutative ring and $n\ge 1$. The following are equivalent:
\begin{enumerate}
\item  Every matrix in $\mathbb M_n(R)$ is a sum of three involutive matrices.
\item $R$ is a subdirect product of $\mathbb Z_3$'s.
\end{enumerate}
\end{thm}

\medskip 
Throughout, rings $R$ are associative with $1$.  
An element $a$ of a ring is an idempotent if $a^2=a$, and is an involution or an involutive element if $a^2=1$.
For a ring $R$,  the Jacobson radical, the set of units, the set of nilpotents, the set of idempotents, and the set of involutions of a ring $R$ are denoted by $J(R)$, $U(R)$, ${\rm Nil}(R)$, ${\rm idem}(R)$ and ${\rm invo}(R)$, respectively. As usual,  ${\mathbb M}_n(R)$ stands for the $n\times n$ matrix ring over $R$ whose identity is denoted by $I_n$.  For a matrix $A$, the trace and rank of $A$ are denoted by ${\rm tr}(A)$ and ${\rm rank}(A)$, respectively. For a positive integer $n$, we write $\mathbb Z_n$ for the ring of integers modulo $n$.

\section{Matrices over a field}

\begin{lem}\label{lem2.1}
If $-1$ is a sum of three idempotents in a ring $R$, then $2^2\cdot 3\cdot 5=0$ in $R$.
\end{lem}
\begin{proof}
Write $-1=e+f+g$ where $e,f,g$ are idempotents in $R$. Then $1+3e=(-1-e)^2=(f+g)^2=f+g+fg+gf=(-1-e)+fg+gf$, so
\begin{equation*}
2+4e=fg+gf.
\end{equation*}
It follows that $2+4(-1-f-g)=fg+gf$. That is, 
\begin{equation*}
-2-4f-4g=fg+gf.
\end{equation*}
So $f(-2-4f-4g)=f(fg+gf)$, i.e., $-6f-5fg=fgf$. Moreover, $(-2-4f-4g)f=(fg+gf)f$, i.e., $-6f-5gf=fgf$. It follows that
\begin{equation*}
5fg=5gf.
\end{equation*}
Thus, $5(-6f-5fg)=5fgf=f(5gf)=f(5fg)=5fg$. So $-30f=30fg$, and hence $-30fg=30fg$, i.e., $60fg=0$. It follows that $60f=0$. Similarly, $60e=0$ and $60g=0$. So $60=-60(e+f+g)=0$.
\end{proof}

\begin{lem}\label{lem2.2}
Let $n\ge 1$ and $F$ be a field. If  every invertible matrix in ${\mathbb M}_n(F)$ is a sum of three idempotents, then $F\cong \mathbb Z_p$ where $p=2,3$ or $5$.
\end{lem}
\begin{proof}
By Lemma \ref{lem2.1}, ${\rm ch}(F)=p$ where $p=2,3$ or $5$. For any $0\not= a\in F$, $A:=\begin{pmatrix}a&0\\
                       0&I_{n-1}\end{pmatrix}$ is invertible in $\mathbb M_n(F)$, so $A=E_1+E_2+E_3$ where $E_1,E_2,E_3$ are idempotent matrices over $F$. Thus, ${\rm tr}(A)={\rm tr}(E_1)+{\rm tr}(E_2)+{\rm tr}(E_3)={\rm rank}(E_1)+{\rm rank}(E_2)+{\rm rank}(E_3)\in \mathbb Z\cdot 1_F=\mathbb Z_p$. It follows that $a\in \mathbb Z_p$. Hence $F\cong \mathbb Z_p$.
\end{proof}

The following known result is a consequence of \cite[Theorem 3]{de1} (this generalizes a classical result of Hartwig and Putcha).

\begin{lem}\label{lem2.3}{\rm (}Symmetry Lemma{\rm )}. Let $F$ be a field and $A\in {\mathbb M}_n(F)$ be a sum of two idempotent matrices. Then, for all $\lambda\in F\backslash \{0, 1_F, 2\cdot 1_F\}$, the scalars $\lambda$ and $2\cdot 1_F-\lambda$ have the same {\rm (}algebraic{\rm )} multiplicity as eigenvalues of $A$.
\end{lem}
 
In \cite[p.861]{de}, it was claimed that, for a prime $p\ge 5$, the matrix $(p-1)\cdot I_n$ in $\mathbb M_n(\mathbb Z_p)$ is not a sum of three idempotents.  But, this claim is false:  
If $n=2k$ and $p=5$, then $(p-1)I_n=-I_n=\begin{pmatrix}I_k&0\\
                       0&0\end{pmatrix}+\begin{pmatrix}-I_k&I_k\\
                       -2I_k&2I_k\end{pmatrix}+\begin{pmatrix}-I_k&-I_k\\
                       2I_k&2I_k\end{pmatrix}$ is a sum of three idempotents in $\mathbb M_n(\mathbb Z_p)$.
Indeed we have the following result.											
 
\begin{lem}\label{lem2.4}
Let $p$ be a prime number greater
than $3$, and $n$ be a positive integer. If $-I_n$ is a sum of three idempotent matrices of $\mathbb M_n(\mathbb Z_p)$, then $p = 5$ and $n$ is even. 
\end{lem}
\begin{proof}
Assume that $-I_n = E_1 + E_2 + E_3$ for some idempotents $E_1, E_2, E_3$. Then
$A := -I_n - E_1$ is the sum of two idempotents. On the other hand, $A$ is diagonalizable with
eigenvalues in $\{-\bar 2,-\bar 1\}$. None of those eigenvalues belongs to $\{\bar 0, \bar 1, \bar 2\}$, whence for each such eigenvalue $\lambda$, by Lemma \ref{lem2.3} we find that $\bar 2 -\lambda$ is also an eigenvalue of $A$. It easily follows that $p = 5$. Next, assume that $n$ is odd. Then, we see that the mapping $\lambda\mapsto  \bar 2 - \lambda$ exchanges the two elements of $\{-\bar 2,-\bar 1\}$ whence $-\bar 2$ and $-\bar 1$ have the same multiplicity as eigenvalues of $A$, and we conclude that $n$ is even.
\end{proof}
%Note that the lemma corrects the claim in \cite[Remark (ii) in Section 7]{de} that $−I_n$ is not the sum of three idempotents if the underlying field is $\mathbb Z_5$.

\begin{lem}\label{lem2.5} Let $n = 2m$ be a positive even integer. Then 
$\begin{pmatrix}1&0\\
                       0&-I_{n-1}\end{pmatrix}$
is not the sum of three idempotents of $\mathbb M_n(\mathbb Z_5)$.
\end{lem}
\begin{proof}
Set $A :=\begin{pmatrix}1&0\\
                       0&-I_{n-1}\end{pmatrix}$  and assume that $A = E_1 +E_2 +E_3$ for some idempotent matrices $E_1, E_2, E_3$ in $\mathbb M_n(\mathbb Z_5)$. If the $E_i$'s all had rank $m$, then 
\begin{equation*}
{\rm tr}(A) = 3\overline m = -2\overline m = -\overline n,
\end{equation*}
which is false. Hence, one of the $E_i$'s, say $E_1$, has an eigenvalue $\alpha$ with multiplicity $r$ greater than $m$. By Grassmann’s formula, one finds that $Ker(A + I_n) \cap Ker(E_1 -\alpha I_n)$ has dimension at least
$r-1$, and it follows that the geometric multiplicity of $-1 -\alpha$ as an eigenvalue of $M := A -E_1$ is at least $r -1$. Using the Symmetry Lemma, we deduce that $2(r-1) \le n$, whence $r = m + 1$
and $M$ has characteristic polynomial $(x + \bar 1)^m(x + \bar 2)^m$. Therefore, for some $\epsilon \in \{-1, 1\}$,
\begin{equation*}
{\rm tr}(A) = {\rm tr} (E_1) + {\rm tr}(M) = (m + \epsilon)\bar 1 -m\bar 1 - m\bar 2 = \overline{\epsilon}- \overline {n},
\end{equation*}
which contradicts the fact that ${\rm tr}(A) = -\overline{n} + \overline{2}$.											
\end{proof}
Note that with the same method, one can prove that if $n$ is even then the matrix $\begin{pmatrix}2&0\\
                       0&-I_{n-1}\end{pmatrix}$ is not the sum of three idempotent matrices of $\mathbb M_n(\mathbb Z_5)$.
Here is the main result in this section.

\begin{thm}\label{thm2.6}
Let $F$ be a field and $n\ge 1$. The following are equivalent:
\begin{enumerate}
\item Every matrix in $\mathbb M_n(F)$ is a sum of three idempotents.
\item Every invertible matrix in $\mathbb M_n(F)$ is a sum of three idempotents.
\item  $F\cong \mathbb Z_2$ or $F\cong \mathbb Z_3$.
\end{enumerate}
\end{thm}
\begin{proof}
The implication $(1)\Rightarrow (2)$ is clear. The implication $(2)\Rightarrow (3)$ follows 
from Lemmas \ref{lem2.2}, \ref{lem2.4} and \ref{lem2.5}. The implication $(3)\Rightarrow (1)$ follows from \cite[Proposition 9]{de}.
\end{proof}

A strengthening of Theorem \ref{thm2.6} is given below.

\begin{cor}\label{cor2.7}
Let $R$ be an integral domain and $n\ge 1$.  Then every matrix in $\mathbb M_n(R)$ is a sum of three idempotents if and only if  $R\cong \mathbb Z_2$ or $R\cong \mathbb Z_3$.
\end{cor}
\begin{proof} We verify the necessity.  Let $Q$ be the field of quotients of $R$, and let $E^2=E\in \mathbb M_n(R)$.  Then, in ${\mathbb M}_n(Q)$,  $E$ is equivalent to a diagonal matrix. So, by \cite[Theorem 1]{S}, there exists an invertible matrix $U$ in ${\mathbb M}_n(Q)$ such that
$U^{-1}EU=\begin{pmatrix}I_r&0\\
                       0&0\end{pmatrix}$ where $r={\rm rank}(E)$. As the trace is similarity-invariant, we have ${\rm tr}(E)={\rm rank}(E)\cdot 1_Q={\rm rank}(E)\cdot 1_R\in {\mathbb Z}\cdot 1_R$. Now let $a\in R$ and write $\begin{pmatrix}a&0\\
                       0&0\end{pmatrix}=E_1+E_2+E_3$ where each $E_i$ is an idempotent matrix in ${\mathbb M}_n(R)$. Then $a={\rm tr}(E_1+E_2+E_3)={\rm tr}(E_1)+{\rm tr}(E_2)+{\rm tr}(E_3)={\rm rank}(E_1)\cdot 1_R+{\rm rank}(E_2)\cdot 1_R+{\rm rank}(E_3)\cdot 1_R=\big({\rm rank}(E_1)+{\rm rank}(E_2)+{\rm rank}(E_3)\big)\cdot 1_R\in \mathbb Z\cdot 1_R$. So, $R=\mathbb Z\cdot 1_R$. As $R$ is a domain and $2^2\cdot 3\cdot 5=0$ in $R$ (Lemma \ref{lem2.1}), we deduce that ${\rm ch}(R)=p$ where $p=2, 3$ or $5$. It follows that $R\cong \mathbb Z_p$ where $p=2,3$ or $5$. So, by Theorem \ref{thm2.6}, $R\cong \mathbb Z_2$ or $R\cong \mathbb Z_3$.
\end{proof}

For the integral domain $R=\mathbb Z_2[x]$, $U(R)=\{1\}$. Hence every unit of $R$ is a sum of three idempotents, but $R\not\cong \mathbb Z_2$ and $R\not\cong \mathbb Z_3$. However, we have the following.

\begin{cor}\label{cor2.8}
Let $R$ be an integral domain and $n\ge 2$.  Then every invertible matrix in $\mathbb M_n(R)$ is a sum of three idempotents if and only if $R\cong \mathbb Z_2$ or $R\cong \mathbb Z_3$.
\end{cor}
\begin{proof} We verify the necessity. Let $Q$ be the field of quotients of $R$. As seen in proving Corollary \ref{cor2.7}, for any $E^2=E\in \mathbb M_n(R)$,  ${\rm tr}(E)={\rm rank}(E)\cdot 1_R\in {\mathbb Z}\cdot 1_R$. Let $a\in R$. Then  $\begin{pmatrix}\begin{pmatrix}a&1\\
                       1&0\end{pmatrix}&\bf 0\\
                       \bf 0&I_{n-2}\end{pmatrix}\in {\mathbb M}_n(R)$ is invertible. Write $\begin{pmatrix}\begin{pmatrix}a&1\\
                       1&0\end{pmatrix}&\bf 0\\
                       \bf 0&I_{n-2}\end{pmatrix}=E_1+\cdots+E_k$ where each $E_i$ is an idempotent matrix in ${\mathbb M}_n(R)$. Then 
\begin{equation*}
\begin{split}
a+(n-2)\cdot 1_R&={\rm tr}(E_1+\cdots+E_k)={\rm tr}(E_1)+\cdots+{\rm tr}(E_k)\\
&={\rm rank}(E_1)\cdot 1_R+\cdots+{\rm rank}(E_k)\cdot 1_R\\
&=\big({\rm rank}(E_1)+\cdots+{\rm rank}(E_k)\big)\cdot 1_R\in \mathbb Z\cdot 1_R.
\end{split}
\end{equation*} 
So, $a\in \mathbb Z\cdot 1_R$, and hence $R=\mathbb Z\cdot 1_R$. As arguing as in the proof of Corollary \ref{cor2.7}, we see $R\cong \mathbb Z_2$ or $R\cong \mathbb Z_3$.
\end{proof}

\section{Matrices over a commutative ring}

\begin{lem}\label{lem3.1}
Let $R$ be %an integral domain or 
a commutative local ring and $n\ge 1$. If $E^2=E\in {\mathbb  M}_n(R)$, then ${\rm tr}(E)={\rm rank}(E)\cdot 1_R\in {\mathbb Z}\cdot 1_R$. 
\end{lem}
\begin{proof} %First we assume that $R$ is an integral domain.   Let $Q$ be the field of quotients of $R$.  Then, in ${\mathbb M}_n(Q)$,  $E$ is equivalent to a diagonal matrix. So, by \cite[Theorem 1]{S}, there exists an invertible matrix $U$ in ${\mathbb M}_n(Q)$ such that $U^{-1}EU=\begin{pmatrix}I_r&0\\0&0\end{pmatrix}$ where $r={\rm rank}(E)$. As the trace is similarity-invariant, we have ${\rm tr}(E)={\rm rank}(E)\cdot 1_Q={\rm rank}(E)\cdot 1_R\in {\mathbb Z}\cdot 1_R$.

%Next we assume that $R$ is a commutative local ring.
The claim is clearly true for $n=1$. Assume $n>1$.
If $E\in {\mathbb M}_n(J(R))$, then $E=0$, so ${\rm tr}(E)={\rm rank}(E)\cdot 1_R\in \mathbb Z\cdot 1_R$. Assume that
$E=(e_{ij})\notin {\mathbb M}_n(J(R))$, so for some $i,j$, $e_{ij}\in U(R)$.
Then $E$ is equivalent to $\begin{pmatrix}e_{ij}&\cdots \\
                       \vdots&\ddots\\
                       \end{pmatrix}$, which is equivalent to $\begin{pmatrix}1&\bf 0 \\
                       \bf 0&E_1\\
                       \end{pmatrix}$, where $E_1$ is an $(n-1)\times (n-1)$ matrix. By \cite[Theorem 4]{SG},  
$E$ is similar to a block diagonal  matrix $\begin{pmatrix}a&\bf 0 \\
                       \bf 0&E_2\\
                       \end{pmatrix}$ where $a^2=a$ and $E_2$ is an $(n-1)\times (n-1)$ idempotent matrix.
Therefore, $a=0$ or $1$ and ${\rm tr}(E_2)={\rm rank}(E_2)\cdot 1_R\in \mathbb Z\cdot 1_R$ by induction assumption.
As similarity preserves trace and rank of matrices over commutative local rings, we have
${\rm tr}(E)={\rm tr}\begin{pmatrix}a&\bf 0 \\
                       \bf 0&E_2\\
                       \end{pmatrix}=a+{\rm tr}(E_2)=a+{\rm rank}(E_2)\cdot 1_R={\rm rank}\begin{pmatrix}a&\bf 0 \\
                       \bf 0&E_2\\
                       \end{pmatrix}\cdot 1_R={\rm rank}(E)\cdot 1_R
											\in \mathbb Z\cdot 1_R$. 
\end{proof}

\begin{thm}\label{thm3.2} Suppose that every matrix in $\mathbb M_n(R)$ is a sum of three idempotents where 
$R$ is a commutative ring and $n\ge 1$.  Then $J(R)$ is nil and $R/J(R)$ has identity $x^3=x$.
If in addition $R$ is an indecomposable ring, then  $R\cong \mathbb Z_n$ where $n=2,3,$ or $4$.
\end{thm}
\begin{proof}
As $R$ is commutative, ${\rm Nil}(R)$ is an ideal of $R$, and $R/{\rm Nil}(R)$ is reduced. So $R/{\rm Nil}(R)$ is a subdirect product
of integral domains $\{R_\alpha\}$. As $\mathbb M_n(R_\alpha)$ is a homomorphic image of  $\mathbb M_n(R)$, every matrix in $\mathbb M_n(R_\alpha)$ is a sum of three idempotents. 
So, by Corollary \ref{cor2.7}, $R_\alpha$ is isomorphic to either $\mathbb Z_2$ or $\mathbb Z_3$. This shows that each $R_\alpha$ has identity $x^3=x$, so $R/{\rm Nil}(R)$ has identity $x^3=x$. It easily follows that $J(R)={\rm Nil}(R)$. 

Suppose that $R$ is indecomposable. Let $a\in R\backslash J(R)$. Then $\bar a^3=\bar a$, so $a^4-a^2\in J(R)$. As $J(R)$ is nil, idempotents lift modulo $J(R)$, so $a^2-e\in J(R)$ for some $e^2=e\in R$. If $e=0$, then $a=(a-a^2)+a^2\in J(R)$, a contradiction. So, $e=1$, and hence  $a^2-1\in J(R)$. This shows that $a\in U(R)$. Hence, we have shown that $R$ is a local ring. For any $a\in R$,  
$\begin{pmatrix}a&0\\
                       0&I_{n-1}\end{pmatrix}$  is a sum of three idempotents. We deduce by Lemma \ref{lem3.1} that ${\rm tr}\begin{pmatrix}a&0\\
                       0&I_{n-1}\end{pmatrix}\in \mathbb Z\cdot 1_R$. So $a\in \mathbb Z\cdot 1_R$, and hence $R=\mathbb Z\cdot 1_R$.  As $2^2\cdot 3\cdot 5=0$ in $R$ (by Lemma \ref{lem2.1}), the Chinese Remainder Theorem ensures that $R=A\times B\times C$ where $2^2=0$ in $A$, $3=0$ in $B$ and $5=0$ in $C$. As $R$ is indecomposable, $R=A$ or $R=B$ or $R=C$, which implies $R\cong \mathbb Z_n$ where $n=2,3,4$ or $5$. But $n=5$ is ruled out by Theorem \ref{thm2.6}. 
\end{proof}

\begin{lem}\label{lem3.3}
Let $R$ be a ring and $n\ge 1$. Then every element of $\mathbb M_n(R)$ is a sum of three idempotents if and only if every element of $\mathbb M_n(R/I)$ is a sum of three idempotents for all indecomposable factor rings $R/I$ of $R$.
\end{lem}
\begin{proof} The necessity is clear. For the sufficiency,  
assume on the contrary that some $(a_{ij})\in {\mathbb M}_n(R)$ is not a sum of three idempotent matrices.  Then 
\begin{equation*}
{\mathcal F}=\big\{I\triangleleft R: \big(\overline{a_{ij}}\big)\in {\mathbb M}_n(R/I)\,\,\,{\text{is not a sum of three idempotents}}\big\}
\end{equation*}
is not empty. For a chain $\{I_{\lambda}\}$ of elements of $\mathcal F$, let $I=\cup _{\lambda}I_\lambda$. Then $I$ is an ideal of $R$. Assume that $\big(\overline{a_{ij}}\big)\in {\mathbb M}_n(R/I)$
is a sum of three idempotents. Then there exist $(\overline{e_{ij}}), (\overline{f_{ij}}), (\overline{g_{ij}})\in {\mathbb M}_n(R/I)$ such that 
\begin{equation}
\begin{split}
&(\overline{a_{ij}})=(\overline{e_{ij}})+(\overline{f_{ij}})+(\overline{g_{ij}}),\\
&(\overline{e_{ij}})(\overline{e_{ij}})=(\overline{e_{ij}}),\,\,(\overline{f_{ij}})(\overline{f_{ij}})=(\overline{f_{ij}}),\,\, (\overline{g_{ij}})(\overline{g_{ij}})=(\overline{g_{ij}}). 
\end{split}
\end{equation} 
Thus, all the following elements are in ${\mathbb M}_n(I)$:
\begin{equation*}
\begin{split}
&(a_{ij})-(e_{ij})-(f_{ij})-(g_{ij}),\\
&(e_{ij})-(e_{ij})(e_{ij}),\,\,(f_{ij})-(f_{ij})(f_{ij}),\,\,(g_{ij})-(g_{ij})(g_{ij})(w_{ij}).
\end{split}
\end{equation*}
Because $\{I_\lambda\}$ is a chain, there exists some $I_\lambda$ such that all these elements are in ${\mathbb M}_n(I_\lambda)$.  
Hence (3.1) holds in ${\mathbb M}_n(R/I_\lambda)$. So, $\big(\overline{a_{ij}}\big)\in {\mathbb M}_n(R/I_\lambda)$
is a sum of three idempotents. This contradiction shows that $I$ is in $\mathcal F$. So $\mathcal F$ is  an inductive set. By Zorn's Lemma, $\mathcal F$ has a maximal element, say $I$. 
We next show that $R/I$ is indecomposable. In fact, if $R/I$ is decomposable, then there exist ideals $I_1, I_2$ of $R$ such that $I\subsetneqq I_k\subsetneqq R$ ($k=1,2$), $R=I_1+I_2$ and $I_1\cap I_2=I$.  So we have the isomorphism
\begin{equation*}
R/I\cong R/I_1\times R/I_2 \,\,\,{\text{via}}\,\,\,r+I\mapsto (r+I_1, r+I_2),
\end{equation*}
which induces an isomorphism 
\begin{equation*}
{\mathbb M}_n(R/I)\cong {\mathbb M}_n(R/I_1)\times {\mathbb M}_n(R/I_2) \,\,\,{\text{via}}\,\,\,(r_{ij}+I)\mapsto \big((r_{ij}+I_1), (r_{ij}+I_2)\big).
\end{equation*}
By the maximality of $I$, $\big(\overline{a_{ij}}\big)\in {\mathbb M}_n(R/I_k)$
is a sum of three idempotents for $k=1,2$. It follows that $\big(\overline{a_{ij}}\big)\in {\mathbb M}_n(R/I)$
is a sum of three idempotents. This contradiction shows that $R/I$ is indecomposable.  But by the hypothesis, every matrix in ${\mathbb M}_n(R/I)$ is a sum of three idempotents, contradicting that $I\in \mathcal F$. 
\end{proof}

\begin{thm}\label{thm3.4}
Let $R$ be a commutative ring with ${\rm Nil}(R)=0$ {\rm (}e.g., $J(R)=0${\rm )} and $n\ge 1$.  The following are equivalent:
\begin{enumerate}
\item Every matrix in $\mathbb M_n(R)$ is a sum of three idempotents.

%\item Every invertible matrix in $\mathbb M_n(R)$ is a sum of three idempotents.
\item $R\cong A\times B$, where $A$ is a Boolean ring and $B$ is zero or a subdirect product of $\mathbb Z_3$'s.
\item $R$ has the identity $x^3=x$.
\end{enumerate}
\end{thm}
\begin{proof}
The implication $(1)\Rightarrow (3)$ follows from Theorem \ref{thm3.2}.  The equivalence $(2)\Leftrightarrow (3)$ is easily seen. 

$(3)\Rightarrow (1)$. Let $R'$ be an indecomposable factor ring of $R$. Then $R'$ has identity $x^3=x$.
For any $0\not= a\in R'$, $a^2$ is a nonzero idempotent of $R'$, so $a^2=1$. Thus $R'$ is a field, and 
it easily follows that $R'$ is isomorphic to $\mathbb Z_2$ or $\mathbb Z_3$.
Hence, by Lemma \ref{lem3.3} and Theorem \ref{thm2.6}, every matrix in $\mathbb M_n(R)$ is a sum of three idempotents.
\end{proof}

\begin{exa}\label{exa3.5}  The matrix $\begin{pmatrix}1&1 \\
                       1&0\\
                       \end{pmatrix}\in {\mathbb M}_2(\mathbb Z_4)$ is not a sum of three idempotents.
\end{exa}
\begin{proof}Let $a=\begin{pmatrix}1&1 \\
                       1&0\\
                       \end{pmatrix}\in {\mathbb M}_2(\mathbb Z_4)$. Assume that $a=e+f+g$ is a sum of
three idempotents. 
We first see that $e\not= 0$. In fact, if $e=0$ then $f$ and $g$ are non-trivial idempotents, so ${\rm rank}(f)={\rm rank}(g)=1$; hence $1={\rm tr}(a)={\rm tr}(f)+{\rm tr}(g)=({\rm rank}(f)+{\rm rank}(g))\cdot 1_{\mathbb Z_4}=2$, a contradiction. So $e\not= 0$. Similarly, $f\not= 0$ and $g\not= 0$. We next see that $e\not= 1$. In fact, if $e=1$ then $f$ and $g$ are non-trivial idempotents, so ${\rm rank}(f)={\rm rank}(g)=1$; hence $-1={\rm tr}(a-e)={\rm tr}(f)+{\rm tr}(g)=({\rm rank}(f)+{\rm rank}(g))\cdot 1_{\mathbb Z_4}=2$, a contradiction. So $e\not= 1$. Similarly, $f\not= 1$ and $g\not= 1$. Hence, $e, f, g$ are non-trivial idempotents, so they all have rank $1$. Thus, $1={\rm tr}(a)={\rm tr}(e)+{\rm tr}(f)+{\rm tr}(g)=({\rm rank}(e)+{\rm rank}(f)+{\rm rank}(g))\cdot 1_{\mathbb Z_4}=3$, a contradiction. Therefore, $a$ is not a sum of three idempotents.
\end{proof} 

\begin{prop}\label{prop3.6} The following are equivalent for a commutative ring $R$:
\begin{enumerate}
\item Every matrix in $\mathbb M_2(R)$ is a sum of three idempotents.
\item Every matrix in $\mathbb M_n(R)$ is a sum of three idempotents for all $n\ge 1$.
\item $R\cong A\times B$, where $A$ is a Boolean ring and $B$ is zero or a subdirect product of $\mathbb Z_3$'s.
\end{enumerate}
\end{prop}
\begin{proof}
$(2)\Rightarrow (1)$. This is clear.

$(3)\Rightarrow (2)$. If $(3)$ holds, then $R$ has identity $x^3=x$, and so ${\rm Nil}(R)=0$. 
Thus, $(2)$ follows from Theorem \ref{thm3.4}.

$(1)\Rightarrow (3)$. By Birkhoff Theorem, $R$ is isomorphic to a subdirect product of subdirectly
irreducible rings $\{R_\alpha\}$. For each $\alpha$, $(1)$ holds for ${\mathbb M}_2(R_\alpha)$. As $R_\alpha$ is indecomposable, $R_\alpha\cong \mathbb Z_n$ where $n=2, 3$ or $4$ by Theorem \ref{thm3.2}. But, $n\not= 4$ by Example \ref{exa3.5}, so $R_\alpha\cong \mathbb Z_2$ or $R_\alpha\cong \mathbb Z_3$.
It follows that $R$ is a direct product of a Boolean ring and a subdirect product of $\mathbb Z_3$'s.
\end{proof}

\begin{rem} Let $R$ be a commutative ring and $n\ge 1$. If every matrix in ${\mathbb M}_n(\mathbb Z_4)$ is a sum of three idempotents {\rm (}e.g., $n=1${\rm )}, then every matrix in ${\mathbb M}_n(R)$ is a sum of three idempotents if and only if every indecomposable factor ring of $R$ is isomorphic to $\mathbb Z_2$, $\mathbb Z_3$ or $\mathbb Z_4$. If not every matrix in ${\mathbb M}_n(\mathbb Z_4)$ is a sum of three idempotents {\rm (}e.g., $n=2${\rm )}, then every matrix in ${\mathbb M}_n(R)$ is a sum of three idempotents if and only if every indecomposable factor ring of $R$ is isomorphic to $\mathbb Z_2$ or $\mathbb Z_3$. Therefore, determining when
every matrix in $\mathbb M_n(R)$ is a sum of three idempotents depends on whether every matrix in $\mathbb M_n(\mathbb Z_4)$ is a sum of three idempotents. But we have been unable to identify the integers $n$ such that every matrix in $\mathbb M_n(\mathbb Z_4)$ is a sum of three idempotents.
\end{rem}

We conclude this section with a characterization of rings for which every element is a sum of three commuting idempotents. 
\begin{prop}\label{prop3.8}
The following are equivalent for a ring $R$:
\begin{enumerate}
\item Every element of $R$ is a sum of three commuting idempotents.
\item $R$ is commutative and every element of $R$ is a sum of three idempotents.
\item $R$ is one of the following types:
\begin{enumerate}
\item $R/J(R)$ is Boolean with $J(R)=2\cdot {\rm idem}(R)$ and $4=0$ in $R$.
\item $R$ is a subdirect product of $\mathbb Z_3$'s.
\item $R\cong A\times B$ where $A/J(A)$ is Boolean with $J(A)=2\cdot {\rm idem}(A)$ and $4=0$ in $A$, and $B$ is a subdirect product of $\mathbb Z_3$'s.
\end{enumerate}
\end{enumerate}
\end{prop}
\begin{proof}
$(2)\Rightarrow (1)$. The implication is clear.

$(3)\Rightarrow (2)$. By \cite[Theorem 1]{HT} , $(3)(b)$ implies $(2)$. So it suffices to show that $(3)(a)$ implies $(2)$. Let us assume that $R/J(R)$ is Boolean with $J(R)=2\cdot {\rm idem}(R)$ and with $4=0$ in $R$. Then ${\rm Nil}(R)=J(R)$. For $a\in R$, $a^2-a$ is nilpotent, so by \cite[Lemma 3.5]{YKZ}, there exists $\theta(t)\in \mathbb Z[t]$ such that $\theta(a)^2=\theta(a)$ and $a-\theta(a)$ is nilpotent. As $a-\theta(a)\in J(R)$, $a-\theta(a)=2h$ where $h^2=h$. Hence $a=\theta(a)+h+h$ is a sum of three idempotents. It remains to show that $R$ is commutative. To do so, we only need to show that every idempotent in $R$ is central. Assume on the contrary that $R$ contains a non-central idempotent, say $e$. Then either $eR(1-e)\not= 0$ or $(1-e)Re\not= 0$. Without loss of generality, we can assume that $eR(1-e)\not= 0$, and let us take $0\not= z\in eR(1-e)$. Consider the Peirce decomposition $R=\begin{pmatrix}eRe&eR(1-e)\\
                       (1-e)Re&(1-e)R(1-e)\end{pmatrix}$. As ${\rm Nil}(R)=J(R)$, we have  $J(R)=\begin{pmatrix}eJ(R)e&eR(1-e)\\
                       (1-e)Re&(1-e)J(R)(1-e)\end{pmatrix}$, so $\begin{pmatrix}0&z\\
                       0&0\end{pmatrix}\in J(R)$, and hence $\begin{pmatrix}0&z\\
                       0&0\end{pmatrix}=2h$ for some $h^2=h\in R$. Write $h=\begin{pmatrix}r&x\\
                       y&s\end{pmatrix}$, so $x=rx+xs$. From 
 $\begin{pmatrix}0&z\\
                       0&0\end{pmatrix}=2h$, it follows that $2r=0$, $2s=0$ and $2x=z$. Hence, $z=2(rx+xs)=(2r)x+x(2s)=0$. This is a contradiction. So every idempotent of $R$ is central.
																						
$(1)\Rightarrow (3)$. Assume that every element of $R$ is a sum of three commuting idempotents. 
Write $-1=e+f+g$ where $e,f,g$ are commuting idempotents in $R$. Then $-efg=(e+f+g)efg=3efg$, so $4efg=0$. Thus, $-4ef=4(e+f+g)ef=8ef+4efg=8ef$, showing that $12ef=0$. Similarly, $12eg=0$ and $12fg=0$. Now we have $6=6(e+f+g)^2=6(e+f+g)+12(ef+eg+fg)=-6$, so $12=0$. By the Chinese Remainder Theorem, $R=A\times B$ where $A\cong R/2^2R$ and $B\cong R/3R$, so $4=0$ in $A$ and $3=0$ in $B$. Moreover, $A,B$ both satisfy $(1)$. For $b\in B$, write $b=e+f+g$ where $e,f,g$ are commuting idempotents in $B$. Then $b^2=b+2(ef+eg+fg)$ and $b^3=b+3(b^2-b)+6efg=b$. So $B$ has the identity $x^3=x$, and hence $B$ is either zero or a subdirect product of $\mathbb Z_3$'s (see \cite[Ex.12.11; p200]{L01}).

We can assume that $A\not= 0$. For $a\in A$, write $a=e+f+g$ where $e,f,g$ are commuting idempotents in $A$. Then $a^2=a+2(ef+eg+fg)$, so $a^2-a\in 2A$. Thus $A/2A$ is Boolean. It follows that $J(A)=2A$, and so $A/J(A)$ is Boolean. For $j\in J(A)$, write $-j=e+f+g$ where $e,f,g$ are commuting idempotents in $A$. So $-jef=ef+ef+efg$, showing that  						$efg=(-j+2)ef\in J(A)$. Hence $2efg=0$. Moreover, as $J(A)^2=0$, $0=(-j)^2=-j+2(ef+eg+fg)$, so $j=2(ef+eg+fg)$, and $(ef+eg+fg)^2=(ef+eg+fg)+6efg=ef+eg+fg$. Hence $J(A)=2\cdot {\rm idem}(A)$.		
\end{proof}

\section{Applications: Matrices as the sum of three involutions}

In this section, we will see that $k=3$ is the smallest positive integer such that, for some ring $R$, every matrix over $R$ is a sum of $k$ involutive matrices, and we show that, for a commutative ring $R$ and $n\ge 1$, 
every matrix in ${\mathbb M}_n(R)$ is a sum of three involutive matrices if and only if $R$ is a subdirect product of $\mathbb Z_3$'s.

The next lemma can be easily seen.
\begin{lem}\label{lem4.1}
Let $R$ be a ring with $2\in U(R)$ and $n\ge 1$. Then:
\begin{enumerate}
\item  $e\mapsto 1-2e$ gives a bijection from ${\rm idem}(R)$ to ${\rm invo}(R)$. 
\item  $a\in R$ is a sum of $n$ idempotents if and only if $n-2a$ is a sum of $n$ involutions. 
\item Every element of $R$ is a sum of $n$ idempotents if and only if every element of $R$ is a sum of $n$ involutions. 
\end{enumerate}
\end{lem}

While a ring is Boolean if every element is an idempotent, we easily see that every nonzero element of a ring $R$ is an involution if and only if $R\cong \mathbb Z_2$ or $R\cong \mathbb Z_3$. For $a, k\in R$, if $a^2=k$ we say that $a$ is a $k$-involution.

\begin{thm}\label{thm4.2}
The following are equivalent for a ring $R$:
\begin{enumerate}
\item Every element of $R$ is a sum of two involutions.
\item Every element of $R$ is a sum of three commuting involutions.
\item For some $k\in R$, every element of $R$ is a sum of two $k$-involutions.
\item $R$ is a subdirect product of $\mathbb Z_3$'s.
\end{enumerate} 
\end{thm}
 
\begin{proof}
$(1)\Rightarrow (4)$. Write $1=u+v$ where $u^2=1$ and $v^2=1$. Then $1=(u+v)^2=u^2+v^2+2uv=2+2uv$, so $-1=2uv$. Hence $1=(-1)^2=(2uv)^2=4u^2v^2=4$. This shows that $3=0$ in $R$. Thus, by Lemma \ref{lem4.1}, every element of $R$ is a sum of two idempotents. By \cite[Proposition 6.1]{YKZ}, $R$ is a subdirect product of $\mathbb Z_3$'s.

$(2)\Rightarrow (4)$. Write $2=u+v+w$ where $u, v, w$ are commuting involutions. Then $4=(u+v+v)^2=u^2+v^2+w^2+2(uv+uw+vw)=3+2(uv+uw+vw)$, so $1=2(uv+uw+vw)$. This shows that $2\in U(R)$. Thus, by Lemma \ref{lem4.1}, every element of 
$R$ is a sum of three commuting idempotents.

Write $-1=e+f+g$ where $e,f,g$ are commuting idempotents in $R$. Then $-efg=(e+f+g)efg=3efg$, so $4efg=0$, and hence $efg=0$ (as $2\in U(R)$). Thus, $-ef=(e+f+g)ef=2ef+efg=2ef$, showing that $3ef=0$. Similarly, $3eg=0$ and $3fg=0$. Now we have $3=3(e+f+g)^2=3(e+f+g)+6(ef+eg+fg)=-3$, so $6=0$, and hence $3=0$ in $R$.

For $a\in R$, write $a=e+f+g$ where $e,f,g$ are commuting idempotents in $R$. Then $a^2=a+2(ef+eg+fg)$ and $a^3=a+3(a^2-a)+6efg=a$. So $R$ has identity $x^3=x$, and hence $R$ is either zero or a subdirect product of $\mathbb Z_3$'s (see \cite[Ex.12.11; p200]{L01}).

$(4)\Rightarrow (1)$. Suppose that $(4)$ holds. Then $R$ has identity $x^3=x$, and hence, by \cite[Theorem 1]{HT}, every element of $R$ is a sum of two idempotents. So, for $a\in R$, $-2+a=e+f$ where $e^2=e$ and $f^2=f$. Thus, as $3=0$ in $R$, $a=(e+1)+(f+1)$ is a sum of two involutions.  

$(4)\Rightarrow (2)$. Suppose that $(4)$ holds. Then $R$ has identity $x^3=x$, and hence, by \cite[Theorem 1]{HT}, every element of $R$ is a sum of three commuting idempotents. So, for $a\in R$, $a=e+f+g$ where $e^2=e$, $f^2=f$ and $g^2=g$. Thus, as $3=0$ in $R$, $a=(e+1)+(f+1)+(g+1)$ is a sum of three commuting involutions. 

$(1)\Rightarrow (3)$. The implication is clear.

$(3)\Rightarrow (1)$. Write $-1=a+b$ where $a^2=b^2=k$. Then $1=(a+b)^2=a^2+b^2+2ab=2k+2ab=2(k+ab)$, so $2\in U(R)$ and $2ab=1-2k$. Thus, $(1-2k)^2=4a^2b^2=4k^2$, showing that $k=4^{-1}$. 
Write $2=u+v$ where $u^2=v^2=k$. Then $4=(u+v)^2=u^2+v^2+2uv=2^{-1}+2uv$, so $8=1+4uv$, i.e., $7=4uv$. It follows that $49=16u^2v^2=1$, or $48=0$. As $2\in U(R)$, it must be that $3=0$, so $k=4^{-1}=1$.  
\end{proof}

By Theorem \ref{thm4.2}, for any ring $R$ and any $n\ge 2$, some matrix in ${\mathbb  M}_n(R)$ cannot be a sum of two involutive matrices. On the other hand, every matrix in ${\mathbb M}_n(\mathbb Z_3)$ is a sum of three idempotent matrices. So, by Lemma \ref{lem4.1},  every  matrix in ${\mathbb M}_n(\mathbb Z_3)$ is a sum of three involutive matrices. 
Hence,  $k=3$ is the smallest positive integer such that, for some ring $R$, every matrix over $R$ is a sum of $k$ involutive matrices.
Next, as an application of what proved in previous sections, we determine commutative rings over which every matrix is a sum of three involutive matrices.

\begin{thm}\label{thm4.3}
Let $R$ be a commutative ring and $n\ge 1$. The following are equivalent:
\begin{enumerate}
\item  Every matrix in $\mathbb M_n(R)$ is a sum of three involutive matrices.
\item  Every indecomposable factor ring of $R$ is isomorphic to $\mathbb Z_3$.
\item $R$ is a subdirect product of $\mathbb Z_3$'s.
\end{enumerate}
\end{thm}
\begin{proof} 
$(1)\Rightarrow (3)$. Let $R'$ be an indecomposable factor ring of $R$ and $F$ be a field that is a factor ring of $R'$. We first show that $2\not= 0$ in $F$. Assume that $2=0$ in $F$. Let $A$ be an involutive matrix in $\mathbb M_n(F)$.  It is well-known (see \cite[p.192]{J85}) that $A$ is similar to its rational canonical form 
$B=\begin{pmatrix}B_1&&0\\
                       &\ddots&\\
                       0&&B_s\end{pmatrix}$, where $s\ge 1$, $B_i$ is a companion matrix of size $n_i$ and $n_1+ n_2+\cdots+n_s=n$.
As $A$ is involutive, $B$ must be involutive, so each $B_i$ is involutive.
It is easily verified that, if $C$ is
an involutive companion matrix over a field, then $C$ has size $1\times 1$, or $C=\begin{pmatrix}
                        0&1\\
                       1&0\end{pmatrix}$. Thus, we have either $n_i=1$ or $n_i=2$.  So we can assume that, for some $k$, $n_i=1$ for $i=1, \ldots, k$ and $n_i=2$ for $i=k+1, \ldots, s$. Moreover, ${\rm tr}(A)={\rm tr}(B)={\rm tr}(B_1)+\cdots+{\rm tr}(B_k)=k\cdot 1_F$. 
If $n$ is even, then $k$ is even, so ${\rm tr}(A)=0$; If $n$ is odd, then $k$ is odd, so ${\rm tr}(A)=1_F$. Thus, for even $n$, every involutive matrix in ${\mathbb M}_n(F)$ has trace $0$ and that for odd $n$, every involutive matrix in ${\mathbb M}_n(F)$ has trace $1_F$.  
Therefore, for even $n$, $E_{11}$ is not a sum of (three) involutive matrices, and for odd $n>1$, $E_{11}+E_{22}$ is not   a sum of three    
involutive matrices. Moreover, in the case $n=1$, 
$1\in \mathbb Z_2$ is not a sum of three involutions. 
We have proved that $2\not= 0$ in $F$. As every matrix in $\mathbb M_n(F)$ is a sum of three involutive matrices, every matrix in $\mathbb M_n(F)$ is a sum of three idempotent matrices by Lemma \ref{lem4.1}. So $F\cong \mathbb Z_3$ by Theorem \ref{thm2.6}. Hence, for every matrix ideal $M$ of $R'$, $3\in M$, and it follows that $3\in J(R')$. So $2\in U(R')$. 
As every matrix in $\mathbb M_n(R')$ is a sum of three involutive matrices, every matrix in $\mathbb M_n(R')$ is a sum of three idempotent matrices by Lemma \ref{lem4.1}. By Theorem \ref{thm3.2}, $R'\cong \mathbb Z_3$ (as $2\in U(R')$). By Birkhoff's Theorem, $R$ is a subdirect product of subdirectly irreducble rings $\{R_\alpha\}$. For each $\alpha$, $R_\alpha$ is indecomposable, so $R_\alpha\cong \mathbb Z_3$. Hence, $R$ is a subdirect product of $\mathbb Z_3$'s. 

$(3)\Rightarrow (2)$. By $(3)$, $R$ has identity $x^3=x$ and $2\in U(R)$. Let $S$ be  an indecomposable factor ring of $R$. Then 
$S$ has identity $x^3=x$ and $2\in U(S)$. For any $0\not= a\in S$, $a^2$ is a nontrivial idempotent, so $a^2=1$. Thus, $S$ is a field, and it follows that $S\cong \mathbb Z_3$.

$(2)\Rightarrow (1)$. Suppose that $(2)$ holds. Then $2\in U(R)$ and, By Theorem \ref{thm2.6} and Lemma \ref{lem3.3}, every matrix in $\mathbb M_n(R)$ is a sum of three idempotent matrices. Hence, by Lemma \ref{lem4.1}, 
every matrix in $\mathbb M_n(R)$ is a sum of three involutive matrices. 
\end{proof}

\begin{rem}
The proof of $(1)\Rightarrow (3)$ in Theorem \ref{thm4.3} implies that, if $\mathbb Z_2$ is a factor ring of $R$, then, for any $n, k\ge 1$, there is a matrix in $\mathbb M_n(R)$ that is not a sum of $k$ involutions.  
\end{rem}

\section*{Acknowledgments} 
The authors are grateful to the referee and Professor Mikhail Chebotar for valuable comments and suggestions.  
Part of the work was carried out when Zhou was visiting Guangxi Teachers Education University and Qinzhou University.  He gratefully acknowledges the hospitality from the host institutes. This research was supported by Natural Science Foundation of China (11661014, 11661013, 11461010), Guangxi Science Research and Technology Development Project (1599005-2-13), Guangxi Natural Science Foundation (2016GXNSFDA380017), the Scientific Research Fund of Guangxi Education Department (KY2015ZD075), and a Discovery Grant from NSERC of Canada.

\end{document}